\documentclass{article}
\usepackage{amsmath,amssymb}

\newtheorem{theorem}{Theorem}[section]
\newtheorem{lemma}[theorem]{Lemma}

\newtheorem{corollary}[theorem]{Corollary}

\newcommand{\id}{\mbox{id}}

\newcommand{\Sym}{\operatorname{Sym}}

\newenvironment{proof}{\par\noindent{\bf Proof.}}{$\qed$\par\bigskip}
\newcommand{\qed}{\enspace\vrule  height6pt  width4pt  depth2pt}
\begin{document}

\title{Retractability of set theoretic solutions of the
  Yang-Baxter equation
  \thanks{
Research  partially supported by grants of MICIN-FEDER (Spain)
MTM2008-06201-C02-01, Generalitat de Catalunya 2005SGR00206,
Onderzoeksraad of Vrije Universiteit Brussel, Fonds voor
Wetenschappelijk Onderzoek (Belgium), Flemish-Polish bilateral
agreement BIL2005/VUB/06 and MNiSW research grant N201 004 32/0088
(Poland).\newline  2000 Mathematical Subject Classification: Primary
81R50, Secondary 20B25, 20F38, 20B35, 20F16, 20F29.\newline
Keywords: Yang Baxter equation, set theoretic solution,
multipermutation solution, permutation group, group of $I$-type.}}

\author{Ferran Ced\'o \and Eric Jespers \and Jan Okni\'nski}
\date{}
\maketitle

\begin{abstract}

It is shown that square free set theoretic involutive
non-de\-ge\-nerate  solutions of the Yang-Baxter equation whose
associated permutation group (referred to as  an involutive
Yang-Baxter group) is abelian are retractable in the sense of
Etingof, Schedler and Soloviev. This solves a problem of
Gateva-Ivanova in the case of abelian IYB groups. It also implies
that the corresponding finitely presented abelian-by-finite groups
(called the structure groups) are poly-${\mathbb Z}$ groups.
Secondly, an example of a solution with an abelian involutive
Yang-Baxter group which is not a generalized twisted union is
constructed. This answers  in the negative another problem of
Gateva-Ivanova.  The constructed solution is of multipermutation
level $3$. Retractability of solutions is also proved in the case
where the natural generators of the IYB group are cyclic
permutations. Moreover, it is shown that such solutions are
generalized twisted unions.
\end{abstract}

\section{Introduction}

In order to find new solutions of the Yang-Baxter equation,
Drinfeld, in \cite{drinfeld}, posed the question of finding the
simplest possible solutions, the so called set theoretic solutions
on a finite set $X$ (see precise definition below). There are many
papers in this area and with many links to different topics. For a
detailed background and references we refer the reader to
\cite{eting,gat-sol,gateva-set,gateva-matched,gateva-quantum,michel}.
We mention a few highlights. Etingof, Schedler and Soloviev
\cite{eting} and independently Gateva-Ivanova and Van den Bergh
\cite{michel} gave a group theoretical interpretation of the set
theoretic involutive non-degenerate solutions of the Yang-Baxter
equation. In the latter paper it was then shown that the associated
group and semigroup algebras share many homological properties with
commutative polynomial algebras in finitely many variables, see also
\cite{jes-okn-Itype}. Rump in \cite{rump} proved that if such
solutions are square free  (on a set with more than one element)
then they are decomposable, hence confirming a conjecture of
Gateva-Ivanova. A detailed account on these aspects can be found in
\cite{bookspringer}. Very recent papers
\cite{gateva-set,gateva-matched,gateva-quantum} focus on  various
specific constructions of set theoretic solutions, already
introduced in \cite{eting}.  The aim is to show that many solutions
can be built recursively from solutions constructed on smaller sets.

We now first  recall the precise definitions, notations and
results needed to clearly state the problems tackled in this
paper.

 Let $X=\{ x_1,x_2,\dots ,x_n\}$, with $n>1$. Recall that a
set theoretic involutive  non-degenerate solution of the Yang-Baxter
equation on $X$ is a pair $(X,r)$, where $r$ is a map $r\colon
X^2\longrightarrow X^2$ such that:
\begin{itemize}
\item[(1)] $r^2=\mathrm{id}_{X^2}$;
\item[(2)] for $i,j\in\{ 1,\dots ,n\}$ there exist unique $k,l\in
\{ 1,\dots ,n\}$ such that
$r(x_i,x_k)=(x_j,x_l)$;
\item[(3)] $r_{12}\circ r_{23}\circ r_{12}=r_{23}\circ r_{12}\circ
r_{23}$, where $r_{ij}\colon X^3\longrightarrow X^3$ is the map
acting as $r$ on the $(i,j)$ components (in this order) and as the
identity on the remaining component.
\end{itemize}
Such a solution $(X,r)$ is called square free if $r(x_{i},x_{i})
=(x_{i},x_{i})$ for every $i$. Condition (2) implies that the maps
$\sigma_i, \gamma_{i}\colon X\longrightarrow X$ defined by
$r(x_i,x_k)=(x_{\sigma_i(k)},x_{\gamma_{k}(i)})$ are bijective.
Denote by $G_r$ the subgroup $\langle\sigma_1,\dots
,\sigma_n\rangle$ of the symmetric group $\Sym_{n}$. Following
\cite{CJR}, we call the group $G_r$ the involutive Yang-Baxter
group (IYB group) associated to the solution $(X,r)$. Note also
that $G_{r}=\langle \gamma_{1},\ldots ,\gamma_{n}\rangle$ (see for
example \cite{CJR}). It is known that $G_r$ is solvable (see
\cite[Theorem 2.15]{eting}). Conversely, it remains an open
problem to decide which solvable finite groups are IYB groups. In
\cite{CJR} this has been proved for several classes of groups, in
particular, nilpotent finite groups of class $2$ and thus for
abelian finite groups.

Let $G(X,r)$ be the group defined by the presentation
$$\langle x_{1},\ldots, x_{n} \mid x_{i}x_{j}=
x_{k}x_{l} \mbox{ if } r(x_{i},x_{j})=(x_{k},x_{l})\rangle .$$ This
group is called the structure group of the solution $(X,r)$, or the
group of $I$-type associated to the solution $(X,r)$,
\cite{michel,bookspringer}. It is known that this group is
isomorphic with the subgroup of the semidirect product
$Fa_{n}\rtimes G_{r}$  of the free abelian group $Fa_{n}=\langle
u_{1},\ldots ,u_{n}\rangle$ of rank $n$, with $G_{r}$ acting by
$\sigma_{i}(u_{j})=u_{\sigma_{i}(j)}$, generated by the set $\{
(u_{i},\sigma_{i} ) \mid i=1,\ldots ,n\}$. It is known that,
identifying $x_{i}$ with $(u_{i},\sigma_{i})$, $$G(X,r)= \{
(a,\sigma_{a}) \mid a\in Fa_{n} \} \subseteq Fa_{n}\rtimes G_{r},$$
where $a\mapsto {\sigma_{a}}$ is a mapping from $Fa_{n}$ to $G_{r}$
and $\sigma_{i}=\sigma_{u_{i}}$.

In particular,  $G(X,r)$ is a solvable abelian-by-finite group. It
is also a torsion free group (\cite[Corollary 1.4]{michel} or
\cite[Corollary~8.2.7]{bookspringer}). The earlier mentioned
homological properties (see \cite{michel}) of the group algebra
$K[G(X,r)]$ then yield that this algebra is a domain (this also
follows from a result of Brown, see \cite[Theorem 13.4.1]{pasbook}).

The relation $\sim $ on the set $X$, introduced in \cite{eting},
and defined by $x_{i}\sim x_{j}$ if $\sigma_{i}=\sigma_{j}$ is
called the retract relation on $X$. There is a natural induced
solution $Ret(X,r)=(X/\!\sim, \tilde{r})$, and it is called the
retraction of $X$. A solution $(X,r)$ is called a multipermutation
solution of level $m$ if $m$ is the smallest nonnegative integer
such that the solution $Ret^{m}(X,r)$ has cardinality $1$. Here we
define $Ret^{k}(X,r)=Ret(Ret^{k-1}(X,r))$ for $k>1$. If such an
$m$ exists then one also says that the solution is retractable. In
this case, the group $G(X,r)$ is a poly-$\mathbb Z$ group (see
\cite[Proposition 8.2.12]{bookspringer}). For such groups, this of
course then gives a direct proof for the fact that $K[G(X,r)]$ is
a domain.

Note that there is a natural action of $G_r$ on $X$ defined by
$\sigma (x_i)=x_{\sigma (i)}$.

A set theoretic involutive non-degenerate solution $(X,r)$ is
called a generalized twisted union of solutions $(Y,r_{Y})$ and
$(Z,r_{Z})$ if $X$ is a disjoint union of two $G_r$-invariant
non-empty subsets $Y,Z$ such that for all $z,z'\in Z,y,y'\in Y$ we
have
\begin{eqnarray}\label{twisted1}\sigma_{\gamma_{y}(z)\mid Y} =
\sigma_{\gamma_{y'}(z)\mid Y}\end{eqnarray}
\begin{eqnarray}\label{twisted2}\gamma_{\sigma_{z}(y)\mid Z} = \gamma_{\sigma_{z'}(y)\mid
Z}.\end{eqnarray} Here, to simplify notation, we write
$\sigma_{x}$ for $\sigma_{i}$ if $x=x_{i}$, and similarly for all
$\gamma_{i}$.

 If, moreover, $(X,r)$ is a square free solution,  then  conditions
(\ref{twisted1}) and (\ref{twisted2}) are equivalent to
\begin{eqnarray}\label{twisted3}\sigma_{\sigma_{y}(z)\mid Y} =
\sigma_{z\mid Y}\end{eqnarray}
\begin{eqnarray}\label{twisted4}\sigma_{\sigma_{z}(y)\mid Z} = \sigma_{y\mid
Z},\end{eqnarray} (see \cite[Proposition 8.3]{gat-sol} and its
proof).

The following conjectures were formulated by Gateva-Ivanova in
\cite{gat-sol}.

\begin{enumerate}
\item[I)] Every set theoretic involutive non-degenerate
square free solution $(X,r)$ of cardinality $n\geq 2$ is a
multipermutation solution of level $m<n$.
\item[II)] Every multipermutation square free solution of cardinality $n\geq 2$
is a generalized
twisted union.
\end{enumerate}

In Section~\ref{Sect2}  we show that Conjecture I) is true for
solutions with an abelian involutive Yang-Baxter group $G_{r}$.
Actually, we prove more. Namely, that every such solution is
retractable in a stronger sense, obtained by refining the relation
$\sim$ on $X$ by requesting additionally that the elements are in
the same $G_{r}$-orbit on $X$. It follows that the corresponding
structure groups $G(X,r)$ are poly-${\mathbb Z}$ groups. Notice that
this is not true in the case of non square free solutions, as shown
in \cite[Example~8.2.14]{bookspringer}.

In Section~\ref{Sect3} we give an example of a
multipermutation solution of level $3$ with
abelian involutive Yang-Baxter group that is not
a generalized twisted union.  Therefore
Conjecture II) does not hold.

Finally, in Section~\ref{Sect4}, we show that if
every generator $\sigma_{i}$ of the IYB group
$G_{r}$ is a cyclic permutation, then the
corresponding solution $(X,r)$ also is
retractable.  Moreover, we prove that  such
solutions are generalized twisted unions,
provided that  $|X|\geq 2$. As this assumption on
$G_{r}$ does not imply that the group $G_{r}$ is
abelian, this provides another class of solutions
for which Conjecture I) is confirmed.

\section{Solutions with an abelian IYB group} \label{Sect2}

In this section we confirm Conjecture I) for  set theoretic
involutive non-degenerate square free solutions $(X,r)$ with an
associated abelian IYB group.

We will often use the following consequence of
\cite[Theorem~8.1.4, Corollary~8.2.4 and
Theorem~9.3.10]{bookspringer}.

\begin{lemma}  \label{permutat}
Assume that $(X,r)$ is  a set theoretic involutive non-degenerate
square free solution. If $r(x_{i},x_{j})=(x_{k},x_{l})$ for some
$i,j,k,l$ then $\sigma_{i}\circ \sigma_{j}=\sigma_{k}\circ
\sigma_{l}$ in $\Sym_{n}$.
\end{lemma}

By \cite[Theorem 1]{rump}, if $n=|X|>1$ then the number $m$ of
orbits in $X$ under the action of $G_r$ is greater than $1$.  Let
$X_{1},\ldots ,X_{m}$ denote these orbits.

\begin{lemma}\label{moves}
Suppose that $G_r$ is abelian. Let $i,j\in\{ 1,\dots ,n\}$ be such
that $\sigma_i(j)=j$. If $x_j\in X_k$, then $\sigma_i(l)=l$ for
all $x_l\in X_k$.
\end{lemma}

\begin{proof}
Let $x_l\in X_k$. Then there exist $i_1,\dots ,i_t\in \{1,\dots
,n\}$ such that $\sigma_{i_1}\dots\sigma_{i_t}(j)=l$. Since $G_r$ is
abelian, we have
\begin{eqnarray*}
\sigma_i(l)&=&\sigma_i\sigma_{i_1}\dots\sigma_{i_t}(j)\\
&=&\sigma_{i_1}\dots\sigma_{i_t}\sigma_i(j)\\
&=&\sigma_{i_1}\dots\sigma_{i_t}(j)=l.
\end{eqnarray*}
\end{proof}

\begin{corollary}\label{identity}
Suppose that $G_r$ is abelian. Then for all $k$ and for all
$x_i\in X_k$,
$$\sigma_i |_{X_k}=\mathrm{id}_{X_k}.$$
\end{corollary}

\begin{proof}
Since $r$ is square free, $\sigma_i(i)=i$. Thus the result follows
by Lemma~\ref{moves}.
\end{proof}

\begin{lemma}\label{key}
Let $k\in \{ 1,\dots ,n\}$. Suppose that for all $x_i\in X_k$ we
have that $\sigma_i  |_{X_k}=\mathrm{id}_{X_k}$. Let
$x_{i_1},\dots ,x_{i_s}\in X_k$ and $x_{j_1}\in X$.  If
$j_2=\sigma_{i_1}\dots\sigma_{i_s}(j_1)$, then $\sigma_{j_1}
|_{X_k}=\sigma_{j_2} |_{X_k}$.
\end{lemma}

\begin{proof}
Clearly we may assume that $s=1$. Thus
$j_2=\sigma_{i_1}(j_1)$. Then there exists
$x_{i}\in X$ such that
$r(x_{i_1},x_{j_1})=(x_{j_2},x_i)$. By
Lemma~\ref{permutat}, we have that
$\sigma_{i_1}\sigma_{j_1}=\sigma_{j_2}\sigma_{i}$.
Let $x_j\in X_k$. Since $x_{\sigma_{j_1}(j)}\in
X_k$, it follows that
$\sigma_{i_1}(\sigma_{j_1}(j))=\sigma_{j_1}(j)$.
Since $x_{i}=x_{\sigma_{j_2}^{-1}(i_1)}\in X_k$,
we have that
$\sigma_{j_2}\sigma_i(j)=\sigma_{j_2}(j)$.
Therefore $\sigma_{j_1}(j)=\sigma_{j_2}(j)$ and
thus $\sigma_{j_1}|_{X_k}=\sigma_{j_2}|_{X_k}$.
\end{proof}

We say that the solution $r$ is trivial if
$r(x_{i},x_{j})=(x_{j},x_{i})$ for every $i,j$. This is equivalent
to saying that $\sigma_{i}$ is the identity map for every $i$, and
also to the fact that $m=n$.

\begin{theorem}\label{abelian}
Assume that $(X,r)$ is a set theoretic involutive non-de\-ge\-nerate
square free solution and the group $G_r$ is abelian. If $r$ is not
trivial then there exist $i,j\in \{ 1,\dots ,n\}$ such that
$\sigma_i=\sigma_j$, $i\neq j$ and $x_{i},x_{j}\in X_{k}$ for some
$k\in \{ 1,\ldots,m\}.$
\end{theorem}

\begin{proof}
Suppose the assertion does not hold. So, for every $k$, if
$x_{i},x_{j}\in X_{k}$ are distinct then $\sigma_{i}\neq
\sigma_{j}$. Note that if $|X_q|=1$ for all $q\in \{ 1, \dots
,m\}$, then $\sigma_1=\sigma_2=\dots=\sigma_n=\mathrm{id}$, a
contradiction. Thus we may assume that $|X_1|>1$. We shall prove
by induction that for all $j\leq m$, there exists $X_{k_j}$ such
that $X_{k_1}=X_1$, the set $\{ X_{k_1},\dots ,X_{k_j}\}$ has
cardinality $j$ and there exist $x_{i_{1,j}},x_{i_{2,j}}\in
X_{k_j}$ such that
\begin{itemize}
\item[(i)] $x_{i_{1,j}}\neq x_{i_{2,j}}$,
\item[(ii)] for all $p<j$, there exist $x_{j_{p,1}},\dots
,x_{j_{p,t_p}}\in X_{k_p}$ such that
$$i_{2,j}=\sigma_{j_{p,1}}\dots \sigma_{j_{p,t_p}}(i_{1,j}).$$
\end{itemize}

For $j=1$, we take $X_{k_1}=X_1$ and we choose any two different
elements $x_{i_{1,1}},x_{i_{2,1}}\in X_1$.

Suppose that $j>1$ and there exist $X_{k_1},\dots ,X_{k_{j-1}}$ such
that $X_{k_1}=X_1$, the set $\{ X_{k_1},\dots ,X_{k_{j-1}}\}$ has
cardinality $j-1$ and there exist $x_{i_{1,j-1}},x_{i_{2,j-1}}\in
X_{k_{j-1}}$ such that
\begin{itemize}
\item[(i')] $x_{i_{1,j-1}}\neq x_{i_{2,j-1}}$,
\item[(ii')] for all $p<j-1$, there exist $x_{j_{p,1}},\dots
,x_{j_{p,t_p}}\in X_{k_p}$ such that
$$i_{2,j-1}=\sigma_{j_{p,1}}\dots \sigma_{j_{p,t_p}}(i_{1,j-1}).$$
\end{itemize}
By Lemma~\ref{key} and Corollary~\ref{identity},
we have that
$$\sigma_{i_{1,j-1}}|_{X_{k_l}}=\sigma_{i_{2,j-1}}|_{X_{k_l}}$$
for all $1\leq l\leq j-1$. Since
$x_{i_{1,j-1}}\neq x_{i_{2,j-1}}$, we have that
$\sigma_{i_{1,j-1}}\neq\sigma_{i_{2,j-1}}$. Thus
there exists $i\in\{ 1,\dots ,n\}$ such that
$\sigma_{i_{1,j-1}}(i)\neq\sigma_{i_{2,j-1}}(i)$.
Let $k_j$ be the integer such that $x_i\in
X_{k_j}$. Clearly the set $\{ X_{k_1},\dots
,X_{k_j}\}$ has cardinality $j$. Let $i_{1,j}=
\sigma_{i_{1,j-1}}(i)$ and $i_{2,j}=
\sigma_{i_{2,j-1}}(i)$. Then
$x_{i_{1,j}},x_{i_{2,j}}\in X_{k_j}$. Let $q$ be
the order of the permutation
$\sigma_{i_{1,j-1}}$, then $$i_{2,j}=
\sigma_{i_{2,j-1}}(i)=\sigma_{i_{2,j-1}}\sigma_{i_{1,j-1}}^{-1}(i_{1,j})=
\sigma_{i_{2,j-1}}\sigma_{i_{1,j-1}}^{q-1}(i_{1,j}).$$
Note that $x_{i_{2,j-1}},x_{i_{1,j-1}}\in
X_{k_{j-1}}$. Thus, if $j=2$, then $(ii)$ is
satisfied. Suppose that $j>2$. Let $1\leq p<
j-1$. By $(ii')$, there exist $x_{j_{p,1}},\dots
,x_{j_{p,t_p}}\in X_{k_p}$ such that
$$i_{2,j-1}=\sigma_{j_{p,1}}\dots
\sigma_{j_{p,t_p}}(i_{1,j-1}).$$ Hence there
exist $j'_1,\dots ,j'_{t_p}\in\{ 1,\dots, n\}$
such that
$$r(x_{j_{p,t_p}},x_{i_{1,j-1}})=(x_{\sigma_{j_{p,t_p}}(i_{1,j-1})},x_{j'_{t_p}})$$
and
$$r(x_{j_{p,t_p-v}},x_{\sigma_{j_{p,t_p-v+1}}\dots
\sigma_{j_{p,t_p}}(i_{1,j-1})})=(x_{\sigma_{j_{p,t_p-v}}\dots
\sigma_{j_{p,t_p}}(i_{1,j-1})},x_{j'_{t_p-v}}),$$
for all $1\leq v\leq t_p-1$. Note that
$x_{j'_{1}},\dots ,x_{j'_{t_p}}\in X_{k_p}$. By
Lemma~\ref{permutat} it follows that
$$\sigma_{j_{p,t_p}}\sigma_{i_{1,j-1}}=\sigma_{\sigma_{j_{p,t_p}}(i_{1,j-1})}\sigma_{j'_{t_p}}$$
and
$$\sigma_{j_{p,t_p-v}}\sigma_{\sigma_{j_{p,t_p-v+1}}\dots
\sigma_{j_{p,t_p}}(i_{1,j-1})}=\sigma_{\sigma_{j_{p,t_p-v}}\dots
\sigma_{j_{p,t_p}}(i_{1,j-1})}\sigma_{j'_{t_p-v}},$$
for all $1\leq v\leq t_p-1$. Hence
\begin{eqnarray*}
\sigma_{j_{p,1}}\dots \sigma_{j_{p,t_p}}\sigma_{i_{1,j-1}}&=&
\sigma_{j_{p,1}}\dots
\sigma_{j_{p,t_p-1}}\sigma_{\sigma_{j_{p,t_p}}(i_{1,j-1})}\sigma_{j'_{t_p}}\\
&=& \sigma_{j_{p,1}}\dots
\sigma_{j_{p,t_p-2}}\sigma_{\sigma_{j_{p,t_p-1}}\sigma_{j_{p,t_p}}(i_{1,j-1})}\sigma_{j'_{t_p-1}}\sigma_{j'_{t_p}}\\
&=& \dots \, =\sigma_{\sigma_{j_{p,1}}\dots\sigma_{j_{p,t_p}}(i_{1,j-1})}\sigma_{j'_{1}}\dots\sigma_{j'_{t_p}}\\
&=& \sigma_{i_{2,j-1}}\sigma_{j'_{1}}\dots\sigma_{j'_{t_p}}\\
&=& \sigma_{j'_{1}}\dots\sigma_{j'_{t_p}}\sigma_{i_{2,j-1}}.
\end{eqnarray*}
Therefore
$$\sigma_{j_{p,1}}\dots \sigma_{j_{p,t_p}}\sigma_{i_{1,j-1}}(i)=\sigma_{j'_{1}}\dots\sigma_{j'_{t_p}}\sigma_{i_{2,j-1}}(i),$$
that is
$$\sigma_{j_{p,1}}\dots \sigma_{j_{p,t_p}}(i_{1,j})=\sigma_{j'_{1}}\dots\sigma_{j'_{t_p}}(i_{2,j}).$$
Hence, if $q_z$ is the order of the permutation $\sigma_{j'_{z}}$,
then we have that
$$i_{2,j}=\sigma_{j'_{1}}^{q_1-1}\dots\sigma_{j'_{t_p}}^{q_{t_p}-1}\sigma_{j_{p,1}}\dots
\sigma_{j_{p,t_p}}(i_{1,j}),$$ and
$$x_{j'_{1}},\dots,x_{j'_{t_p}},x_{j_{p,1}},\dots ,
x_{j_{p,t_p}}\in X_{k_p}.$$ Thus $(ii)$ is satisfied. In particular,
for $j=m$, there exist $X_{k_1},\dots ,X_{k_{m}}$ such that
$X_{k_1}=X_1$, the set $$\{ X_{k_1},\dots ,X_{k_{m}}\}=\{
X_{1},\dots ,X_{m}\}$$ and there exist $x_{i_{1,m}},x_{i_{2,m}}\in
X_{k_{m}}$ such that
\begin{itemize}
\item[(i)] $x_{i_{1,m}}\neq x_{i_{2,m}}$,
\item[(ii)] for all $p<m$, there exist $x_{j_{p,1}},\dots
,x_{j_{p,t_p}}\in X_{k_p}$ such that
$$i_{2,m}=\sigma_{j_{p,1}}\dots \sigma_{j_{p,t_p}}(i_{1,m}).$$
\end{itemize}
By Lemma~\ref{key} and Corollary~\ref{identity}, we have that
$$\sigma_{i_{1,m}}|_{X_{k_l}}=\sigma_{i_{2,m}}|_{X_{k_l}}$$
for all $1\leq l\leq m$, that is
$\sigma_{i_{1,m}}=\sigma_{i_{2,m}}$, a contradiction, therefore
the assertion follows.
\end{proof}

Motivated by Theorem~\ref{abelian}, we define a relation $\rho$ on
$X$ as follows:
$$(x_{i},x_{j})\in \rho \mbox{ if }
x_{i},x_{j}\in X_{k} \mbox{ for some } k \mbox{ and }
\sigma_{i}=\sigma_{j}.$$ This can be used to define a stronger
version of retractability of $(X,r)$, based on the relation
$\rho$. In order to make this idea work, we need some observations
that are similar to those known for the retract relation $\sim$.

Recall that every  set theoretic involutive non-degenerate square
free solution satisfies the so called cyclic condition. This says
that if $r(x_w,x_{j_{1}}) = (x_{j_{2}},x_{w'})$ then there exist
$j_{3},\ldots ,j_{k}$, such that
$$r(x_w,x_{j_{2}})=(x_{j_{3}},x_{w'}), \ldots
,r(x_{w},x_{j_{k}})=(x_{j_{1}},x_{w'}),$$ see
\cite[Corollary~9.2.6]{bookspringer}.

\begin{lemma}\label{rho}
The relation  $\rho$ is an equivalence relation that is compatible
with $r$. That is, if $r(x_i,x_j)=(x_p,x_q)$ and
$r(x_w,x_v)=(x_k,x_l)$ with $(x_i,x_w),(x_j,x_v)\in \rho$, then
$(x_p,x_k),(x_q,x_l)\in\rho$.
\end{lemma}

\begin{proof}
Since $(x_i,x_w)\in \rho$, we have that
$\sigma_w(j)=\sigma_i(j)=p$. Thus there exists $w'$
such that $r(x_w,x_j)=(x_p,x_{w'})$.  Since
$r^2=\mathrm{id}_{X^2}$, it follows that
$$r(x_p,x_q)=(x_i,x_j)\quad \mbox{and}\quad
r(x_p,x_{w'})=(x_w,x_j).$$ Thus $\sigma_p(q)=i$ and
$\sigma_p(w')=w$. Since $(x_i,x_w)\in \rho$, we thus
obtain that $x_q,x_i,x_{w'},x_w$ are in the same
$G_r$-orbit on $X$. By Lemma~\ref{permutat}, it
follows that $\sigma_i\sigma_j=\sigma_p\sigma_q$ and
$\sigma_w\sigma_j=\sigma_p\sigma_{w'}$. Hence, since
$\sigma_i=\sigma_w$, we have that
$$\sigma_{w'}=\sigma_p^{-1}\sigma_w\sigma_j=
\sigma_p^{-1}\sigma_i\sigma_j=\sigma_q.$$ Therefore
$(x_q,x_{w'})\in \rho$. Since
$r(x_w,x_j)=(x_p,x_{w'})$, by the cyclic condition,
there exists $p'$ such that
$r(x_w,x_{p'})=(x_j,x_{w'})$. Thus $\sigma_j(w')=w$.
Since $r(x_w,x_v)=(x_k,x_{l})$, by the cyclic
condition, there exists $k'$ such that
$r(x_w,x_{k'})=(x_v,x_{l})$. Thus $\sigma_v(l)=w$.
Since $(x_j,x_v)\in \rho$, it follows that
$$l=\sigma_v^{-1}(w)=\sigma_j^{-1}(w)=w'.$$
Therefore $(x_q,x_l)\in \rho$. Since $\sigma_i(j)=p$,
$\sigma_w(v)=k$ and $(x_j,x_v)\in \rho$, the elements
$x_j,x_p,x_v,x_k$ are in the same $G_r$-orbit on $X$. The assumption
$r(x_w,x_v)=(x_k,x_l)$ implies, in view of Lemma~\ref{permutat},
that $\sigma_w\sigma_v=\sigma_k\sigma_l$. Since $\sigma_l=\sigma_q$,
$\sigma_w=\sigma_i$, $\sigma_v=\sigma_j$ and
$\sigma_i\sigma_j=\sigma_p\sigma_q$, this  yields that
$$\sigma_k=\sigma_w\sigma_v\sigma_l^{-1}=\sigma_i\sigma_j\sigma_q^{-1}=\sigma_p.$$
Therefore $(x_p,x_k)\in \rho$.
\end{proof}

By Lemma~\ref{rho}, it is easy to see that the map $\bar r\colon
(X/\rho)^2\longrightarrow (X/\rho)^2$, defined by $\bar
r(\overline{x_i},\overline{x_j})=(\overline{x_k},\overline{x_l})$ if
$r(x_i,x_j)=(x_k,x_l)$, yields a set theoretic involutive
non-degenerate square free solution $(X/\rho, \bar r)$. Let
$|X/\rho|=n'$ and let $X/\rho=\{\overline{x_{i_1}},\dots
,\overline{x_{i_{n'}}}\}$. We denote by $y_j$ the element
$y_j=\overline{x_{i_j}}\in X/\rho$. Let
$\sigma'_j\in\mathrm{Sym}_{n'}$ be the permutation defined by $\bar
r(y_j,y_k)=(y_{\sigma'_j(k)},y_l)$. Let $G_{\bar r}$ denote the
group $\langle \sigma'_1,\dots ,\sigma'_{n'} \rangle$.

\begin{lemma}\label{group}
The map $\phi\colon \{ \sigma_1,\dots ,\sigma_n\}\longrightarrow
\{ \sigma'_1,\dots ,\sigma'_{n'}\}$, defined by
$\phi(\sigma_i)=\sigma'_j$ if $\overline{x_i}=y_j$, extends to a
group epimorphism $\widetilde{\phi}\colon G_r\longrightarrow
G_{\bar r}$.
\end{lemma}

\begin{proof}
Let $\sigma\in G_r$. We define
$\widetilde{\phi}(\sigma)$ by
$$\widetilde{\phi}(\sigma)(k)=j,$$ where
$\overline{x_{\sigma(i_k)}}=y_j$. Note that there
exists $l$ such that
$r(x_i,x_{i_k})=(x_{\sigma_i(i_k)},x_l)$. Thus
$\bar
r(\overline{x_i},\overline{x_{i_k}})=(\overline{x_{\sigma_i(i_k)}},\overline{x_l})$.
Hence if $\overline{x_i}=y_p$ then
$\overline{x_{\sigma_i(i_k)}}=y_{\sigma'_p(k)}$.
Therefore
$$\widetilde{\phi}(\sigma_i)=\sigma'_p=\phi(\sigma_i).$$
So, $\widetilde{\phi}$ is an extension of $\phi$.

Let $\sigma\in G_r$. Then
$\widetilde{\phi}(\sigma)(k)=j$, where
$\overline{x_{\sigma(i_k)}}=y_{j}$. Let $i\in \{
1,\dots, n\}$. Then there exists $i'$ such that
$$r(x_{i},x_{i_j})=(x_{\sigma_{i}(i_j)},x_{i'}).$$
Also there exists $i''$ such that
$$r(x_{i},x_{\sigma(i_k)})=(x_{\sigma_{i}\sigma(i_k)},x_{i''}).$$
Since
$\overline{x_{i_j}}=y_j=\overline{x_{\sigma(i_k)}}$,
by Lemma~\ref{rho}, we have that
$\overline{x_{\sigma_i(i_j)}}=\overline{x_{\sigma_i\sigma(i_k)}}$.
Let
$y_{j'}=\overline{x_{\sigma_i(i_j)}}=\overline{x_{\sigma_i\sigma(i_k)}}$.
Then we have
$$\widetilde{\phi}(\sigma_i)(\widetilde{\phi}(\sigma)(k))
=\widetilde{\phi}(\sigma_i)(j)=j'=\widetilde{\phi}(\sigma_i\sigma)(k).$$
Hence
$$\widetilde{\phi}(\sigma_i)\widetilde{\phi}(\sigma)
=\widetilde{\phi}(\sigma_i\sigma).$$ Thus, by
induction on $s$, it is easy to see that
$$\widetilde{\phi}(\sigma_{k_1})\cdots\widetilde{\phi}(\sigma_{k_s})
=\widetilde{\phi}(\sigma_{k_1}\cdots\sigma_{k_s}).$$
Therefore $\widetilde{\phi}$ is an epimorphism of
groups.
\end{proof}

Note that if $i=\sigma_{k_1}\cdots \sigma_{k_s}(i_k)$ and
$\overline{x_i}=y_j$, then
$$j=\widetilde{\phi}(\sigma_{k_1})\cdots\widetilde{\phi}(\sigma_{k_s})(k).$$
Hence if $x_i$ and $x_{u}$ are in the same orbit under the action of
$G_r$ then $\overline{x_i}$ and $\overline{x_u}$ are in the same
orbit under the action of $G_{\bar r}$. Conversely, if $y_j$ and
$y_k$ are in the same orbit under the action of $G_{\bar r}$, then
there exist $k_1,\dots ,k_s$ such that
$$j=\widetilde{\phi}(\sigma_{k_1})\cdots\widetilde{\phi}(\sigma_{k_s})(k)=
\widetilde{\phi}(\sigma_{k_1}\cdots\sigma_{k_s})(k).
$$
Thus
$\overline{x_{i_j}}=\overline{x_{\sigma_{k_1}\cdots\sigma_{k_s}(i_k)}}$,
and therefore $x_{i_j}$ and
$x_{\sigma_{k_1}\cdots\sigma_{k_s}(i_k)}$ are in the same orbit
under the action of $G_r$. Hence $x_{i_j}$ and $x_{i_k}$ are in the
same orbit under the action of $G_r$. So we have proved the
following result.

\begin{lemma}\label{orbits}
The number of orbits of $X$ under the action of $G_r$ is the same
as the number of orbits of $X/\rho$ under the action of $G_{\bar
r}$. Furthermore, if $X_k$ is an orbit of $X$, then $\{
\overline{x_i}\mid x_i\in X_k\}$ is an orbit of $X/\rho$.
\end{lemma}

Now the notion of strong retractability of
$(X,r)$ may be defined as follows. First, let
$Ret_{\rho}(X,r)=(X/\rho, \bar{r})$ denote the
induced solution. We say that  $(X,r)$ is
strongly retractable if there exists $m\geq 1$
such that applying  $m$ times the operator
$Ret_{\rho}$ we get a trivial solution.

Since the IYB group corresponding to the solution
$(X/\rho, \bar{r})$ also is abelian by
Lemma~\ref{group} if $G_{r}$ is abelian, the
following is a direct consequence of
Theorem~\ref{abelian}.

\begin{corollary}  \label{retract}
Assume that $(X,r)$ is  a set theoretic involutive non-degenerate
square free solution and the group $G_r$ is abelian. Then $(X,r)$ is
strongly retractable.
\end{corollary}

\section{A solution that is not a generalized twisted union}
\label{Sect3}

In this section we present a counterexample to Conjecture II).
Actually, such a construction can be given already in the case where
the group $G_{r}$ is abelian.

\begin{theorem}
There exists a multipermutation square free solution of level $3$
that it is not a generalized twisted union. Furthermore, the
associated IYB group is abelian.
\end{theorem}

\begin{proof}
Let $X=\{ x_i\mid 1\leq i\leq 24\}$. Consider the following
permutations in $\mathrm{Sym}_{24}$,
\begin{eqnarray*}
\sigma_1&=&\sigma_2=(9,10)(11,12)(13,14)(15,16)(17,18)(19,20)(21,22)(23,24),\\
\sigma_3&=&\sigma_4=(9,11)(10,12)(13,15)(14,16)(17,18)(19,20)(21,22)(23,24),\\
\sigma_5&=&\sigma_6=(9,10)(11,12)(13,14)(15,16)(17,19)(18,20)(21,23)(22,24),\\
\sigma_7&=&\sigma_8=(9,11)(10,12)(13,15)(14,16)(17,19)(18,20)(21,23)(22,24),\\
\sigma_9&=&\sigma_{12}=\sigma_{13}=\sigma_{16}=(1,5)(2,6)(3,7)(4,8)(17,21)(18,22)(19,23)(20,24),\\
\sigma_{10}&=&\sigma_{11}=\sigma_{14}=\sigma_{15}=(1,5)(2,6)(3,7)(4,8)(17,24)(18,23)(19,22)(20,21),\\
\sigma_{17}&=&\sigma_{20}=\sigma_{21}=\sigma_{24}=(9,13)(10,14)(11,15)(12,16)(1,3,2,4)(5,7,6,8),\\
\sigma_{18}&=&\sigma_{19}=\sigma_{22}=\sigma_{23}=(9,16)(10,15)(11,14)(12,13)(1,3,2,4)(5,7,6,8).
\end{eqnarray*}

Consider the map $r\colon X^2\longrightarrow X^2$ defined by
$r(x_i,x_j)=(x_{\sigma_i(j)},x_{\sigma^{-1}_{\sigma_i(j)}(i)})$ for
all $i,j\in X$. Note that $r^2=\mathrm{id}_X$ and
$r(x_i,x_i)=(x_i,x_i)$ for all $x_i\in X$. It is well known that in
order to prove that $(X,r)$ is a square free solution it is
sufficient to check that $\sigma_i\sigma_j=\sigma_k\sigma_l$
whenever $r(x_i,x_j)=(x_k,x_l)$ (see
\cite[Theorem~9.3.10]{bookspringer}). This can be checked by a
direct verification.

Consider the following permutations in $\mathrm{Sym}_{24}$,
\begin{eqnarray*}
\tau_1&=&(9,10)(11,12)(13,14)(15,16),\\
\tau_2&=&(9,11)(10,12)(13,15)(14,16),\\
\tau_3&=&(9,13)(10,14)(11,15)(12,16),\\
\tau_4&=&(9,16)(10,15)(11,14)(12,13),\\
\tau_5&=&(17,18)(19,20)(21,22)(23,24),\\
\tau_6&=&(17,19)(18,20)(21,23)(22,24),\\
\tau_7&=&(17,21)(18,22)(19,23)(20,24),\\
\tau_8&=&(17,24)(18,23)(19,22)(20,21),\\
\tau_9&=&(1,5)(2,6)(3,7)(4,8),\\
\tau_{10}&=&(1,3,2,4)(5,7,6,8).
\end{eqnarray*}
It is easy to check that the subgroups $\langle
\tau_1,\tau_2,\tau_3,\tau_4 \rangle$,
$\langle\tau_5,\tau_6,\tau_7,\tau_8\rangle$ and $\langle
\tau_9,\tau_{10}\rangle$ are abelian. Hence, the group $\langle
\tau_{1},\ldots , \tau_{10}\rangle$ is abelian. As $G_{r}=\langle
\sigma_i\mid 1\leq i\leq 24\rangle \subseteq \langle
\tau_{1},\ldots , \tau_{10}\rangle$, we thus obtain that $G_{r}$
is abelian.

It is easy to see that there are three $G_{r}$-orbits on $X$:
\begin{eqnarray*}
X_1 & = & \{ x_1,x_2,x_3,x_4,x_5,x_6,x_7,x_8\} \\
X_2 & = & \{ x_9,x_{10},x_{11},x_{12},x_{13},x_{14},x_{15},x_{16}\} \\
X_3 & = & \{
x_{17},x_{18},x_{19},x_{20},x_{21},x_{22},x_{23},x_{24}\} .
\end{eqnarray*}
Suppose $(X,r)$ is a generalized twisted union $X=Y\cup Z$. Then $Y$
or $Z$ is equal to one of these orbits. Say, $Y=X_{i}$ for some $i$.
Notice that $\sigma_{17}(1)=3$ and $\sigma_{1 \mid X_{2}\cup X_{3}}
\neq \sigma_{3\mid X_{2}\cup X_{3}}$. This implies that $Y\neq
X_{1}$. Similarly, $\sigma_{1}(9)=10$ and $\sigma_{9\mid X_{1}\cup
X_{3}}\neq\sigma_{10\mid X_{1}\cup X_{3}}$ imply that $Y\neq X_{2}$.
Finally, $\sigma_{1}(17)=18$ and $\sigma_{17\mid X_{1}\cup
X_{2}}\neq \sigma_{18\mid X_{21}\cup X_{2}}$. Therefore $Y\neq
X_{3}$. This contradiction shows that $(X,r)$ is not a generalized
twisted union.

It is easy to verify that applying twice the operator $Ret$ we get
a trivial solution of cardinality $3$. Hence, $(X,r)$ is a
multipermutation solution of level $3$.
\end{proof}

\section{IYB groups generated by cyclic permutations} \label{Sect4}

In this section $(X,r)$ will stand for a set theoretic involutive
non-degenerate square free solution with associated IYB group $G_r$,
such that its generators $\sigma_{i}$, $i=1,\ldots ,n$, are cyclic
permutations.

Our aim is to prove that if $|X|>1$ then $(X,r)$ is a retractable
solution, and moreover it is a generalized twisted union. Hence, we
confirm Conjectures I) and II) in this case.

Recall that (see \cite[Corollary~9.2.6 and Proposition
9.2.4]{bookspringer}) every square free solution satisfies the so
called full cyclic condition. This says that  for any distinct
elements $x$ and $y$ in $X$ there exist distinct elements $x=x_1,\,
x_2,\dots , x_k$ and distinct $y=y_1,\, y_2,\dots, y_p$ in $X$ such
that
\begin{eqnarray*}
r(y_1,x_1)=(x_2,y_2),\;
r(y_1,x_2)=(x_3,y_2),\;\ldots ,\,
r(y_1,x_k)=(x_1,y_2),\\ r(y_2,x_1)=(x_2,y_3),\;
r(y_2,x_2)=(x_3,y_3),\;\ldots ,\,
r(y_2,x_k)=(x_1,y_3),\\
\vdots\hphantom{xxxxxxxxxxxxxxxxxxxxxxx}\\
r(y_p,x_1)=(x_2,y_1),\;
r(y_p,x_2)=(x_3,y_1),\;\ldots ,\,
r(y_p,x_k)=(x_1,y_1).
\end{eqnarray*}

Let $X_1,\dots ,X_m$ be the distinct orbits of $X$ under the action
of $G_r$. We denote by $r_j$ the restriction of $r$ on $X_j^2$. Note
that $(X_j,r_j)$ also is a set theoretic involutive non-degenerate
square free solution with associated IYB group $G_{r_j}$, such that
its generators $\sigma_{i\mid X_j}$, for $x_i\in X_j$, are cyclic
permutations.

\begin{lemma}\label{key2}
Let $x_{i_1}\in X_i$. If $\sigma_{i_1\mid X_k}\neq
\mathrm{id}_{X_k}$ for some $k$, then
$$\sigma_{j\mid X_k}\neq\mathrm{id}_{X_k},$$
for all $x_{j}\in X_i$. Furthermore, $\sigma_{i_1},\sigma_j$ are
conjugate elements in $G_r$ for all $x_j\in X_i$.
\end{lemma}
\begin{proof}
Let $x_{i_2}\in X_i$ be an element such that
$x_{i_1}\neq x_{i_2}$. Suppose that there exists
$l\in \{ 1,\dots ,n\}$ such that
$\sigma_{l}(i_1)=i_2$. By the full cyclic
condition there exist distinct elements $x_{i_1},
x_{i_2},\dots , x_{i_k}$ and distinct
$x_{l}=x_{l_1},\, x_{l_2},\dots, x_{l_p}$ in $X$
such that
\begin{eqnarray*}
r(x_{l_1},x_{i_1})=(x_{i_2},x_{l_2}),\;
r(x_{l_1},x_{i_2})=(x_{i_3},x_{l_2}),\;\ldots ,\,
r(x_{l_1},x_{i_k})=(x_{i_1},x_{l_2}),\\
r(x_{l_2},x_{i_1})=(x_{i_2},x_{l_3}),\;
r(x_{l_2},x_{i_2})=(x_{i_3},x_{l_3}),\;\ldots ,\,
r(x_{l_2},x_{i_k})=(x_{i_1},x_{l_3}),\\
\vdots\hphantom{xxxxxxxxxxxxxxxxxxxxxxx}\\
r(x_{l_p},x_{i_1})=(x_{i_2},x_{l_1}),\;
r(x_{l_p},x_{i_2})=(x_{i_3},x_{l_1}),\;\ldots ,\,
r(x_{l_p},x_{i_k})=(x_{i_1},x_{l_1}).
\end{eqnarray*}
Since every $\sigma_{p}$ is a cycle, it follows that
$\sigma_{l}=\sigma_{l_1}=\sigma_{l_2}=(i_1,i_2,\dots, i_k)$. From
Lemma~\ref{permutat} we know that
$\sigma_{l_1}\sigma_{i_1}=\sigma_{i_2}\sigma_{l_2}$, thus
$$\sigma_l\sigma_{i_1}\sigma_{l}^{-1}=\sigma_{i_2}.$$
Since $\sigma_{i_1\mid X_k}\neq
\mathrm{id}_{X_k}$, clearly we have that
$\sigma_{i_2\mid X_k}\neq\mathrm{id}_{X_k}.$

Let $x_j\in X_i$ be an element different from $x_{i_1}$. Since
$x_j,x_{i_1}$ are in the same orbit, there exist $j_1,\dots ,j_t\in
\{ 1,\dots ,n\}$ such that
$$\sigma_{j_1}\cdots\sigma_{j_t}(i_1)=j.$$ Now it
is easy to see by induction on $t$ that $\sigma_{i_1},\sigma_j$ are
conjugate elements in $G_r$, and the result follows.\end{proof}

\begin{theorem}\label{cyclic1} $(X,r)$ is
strongly retractable. Moreover, if $|X|>1$ then
$(X,r)$ is a generalized twisted union.
\end{theorem}

\begin{proof}
We shall prove both statements by induction on $|X|=n$. Clearly,
we may assume that $(X,r)$ is a nontrivial solution. In
particular,  $n>2$. Also, we may assume that the result is true
for all solutions of the same type as $(X,r)$ with cardinality
less than $n$.

Let $X_1,\dots ,X_m$ be the different orbits of $X$ under the action
of $G_r$. As mentioned before,  by \cite[Theorem 1]{rump}, $m>1$.

First assume that $|X_k|=1$ for some $k$. Then, clearly, $(X,r)$
is a generalized twisted union of $X_k$ and $X\setminus X_k$.
Moreover, if $X'=X\setminus X_{k}$ and $r'=r_{|X'\times X'}$ then
the solution $(X',r')$ inherits the assumptions on $(X,r)$ and so
it is strongly retractable by the induction hypothesis.

If the solution $(X',r')$ is nontrivial, this means that there
exist $x_{i},x_{j}\in X', i\neq j$, such that
$\sigma_{i|X'}=\sigma_{j|X'}$ and $x_{i},x_{j}$ are in the same
$G_{r'}$-orbit on $X'$. Then $\sigma_{i}=\sigma_{j}$ and
$x_{i},x_{j}$ are in the same $G_{r}$-orbit on $X$, whence $\rho$
is a nontrivial relation.

If $(X',r')$ is a trivial solution then $\sigma_{i|X'}=\id_{|X'}$
for all $x_{i}\in X'$, so that all $\sigma_{i}, x_{i}\in X'$, are
equal. Suppose that $\rho$ is a trivial relation on $X$. This
implies that every $G_{r}$-orbit contained in $X'$ is of
cardinality $1$. But then $(X,r)$ is a trivial solution, a
contradiction. Therefore the relation $\rho$ is nontrivial also in
this case.

As the induced solution $(X/\rho, \bar{r})$
inherits the assumptions on $(X,r)$, by the
induction hypothesis, it follows that it is
strongly retractable. Therefore $(X,r)$ also is
strongly retractable.

Hence, to complete the proof we may assume that $|X_k|>1$ for all
$k=1,\ldots ,n$.

By \cite[Theorem 1]{rump}, the number of $G_{r_1}$-orbits of $X_1$
is greater than $1$. Therefore there exist $x_{i_1},x_{i_2}\in
X_1$ and $x_j\in X\setminus X_1$ such that $x_{i_1}\neq x_{i_2}$
and $\sigma_j(i_1)=i_2$.  Let $k$ be such that $x_j\in X_k$. By
Lemma~\ref{key2}, $\sigma_{l\mid X_1}\neq \mathrm{id}_{X_1}$ for
all $x_l\in X_k$. Thus $\sigma_{l\mid X_k}= \mathrm{id}_{X_k}$ for
all $x_l\in X_k$. We claim that $\sigma_l=\sigma_j$ for all
$x_l\in X_k$.

Let $x_{l_1}\in X_k$. Suppose first that $l_{1}\ne j$ and
$\sigma_u(j)=l_1$ for some $u\in \{ 1,\dots, n\}$. In this case,
there exists $x_v\in X$ such that $r(x_u,x_j)=(x_{l_1},x_v)$. By
the full cyclic condition, $\sigma_v(j)=l_1$. Thus $\sigma_{u\mid
X_k}\neq \mathrm{id}_{X_k}$ and $\sigma_{v\mid X_k}\neq
\mathrm{id}_{X_k}$. From Lemma~\ref{permutat} we know that
$\sigma_{u}\sigma_{j}=\sigma_{l_1}\sigma_{v}$. Since
$$\sigma_{j\mid X_k}=\sigma_{l_1\mid
X_k}= \mathrm{id}_{X_k}$$ and
$\sigma_u,\sigma_v,\sigma_j,\sigma_{l_1}$ are cyclic permutations,
 this implies that $\sigma_u=\sigma_v$ and
$\sigma_j=\sigma_{l_1}$, as desired. Now it is easy to see that
$\sigma_l=\sigma_j$ for all $x_l\in X_k$, as claimed. In
particular, the relation $\rho$ is nontrivial on $X$.

On the other hand, by Lemma~\ref{key}, for all $x_i\in X\setminus
X_k$ and all $x_l\in X_k$, we have that $\sigma_{i\mid
X_k}=\sigma_{\sigma_l(i)\mid X_k}$. Hence $(X,r)$ is a generalized
twisted union of $X_k$ and $X\setminus X_k$.

Moreover, the induced solution $(X/\rho,\bar{r})$
satisfies all the assumptions on $(X,r)$.
Therefore it is strongly retractable by the
induction hypothesis. It follows that $(X,r)$ is
strongly retractable as well. \end{proof}

Notice that there exist solutions of the above type with nonabelian
groups $G_{r}$, see for example \cite[Example~5.4]{gateva-quantum}.
Therefore, Theorem~\ref{cyclic1} confirms Conjecture I) for a class
of solutions not covered by Theorem~\ref{abelian}.

 \vspace{30pt}
 \noindent \begin{tabular}{llllllll}
 F. Ced\'o && E. Jespers  \\
 Departament de Matem\`atiques &&  Department of Mathematics \\
 Universitat Aut\`onoma de Barcelona &&  Vrije Universiteit Brussel  \\
08193 Bellaterra (Barcelona), Spain    &&
Pleinlaan 2, 1050 Brussel, Belgium \\
   &&   \\
J. Okni\'nski &&  \\ Institute of Mathematics &&
\\ Warsaw University&& \\ Banacha 2&& \\ 02-097
Warsaw, Poland &&
\end{tabular}
\end{document}